\documentclass{amsart}

\usepackage[utf8]{inputenc} 
\usepackage{amsmath, amsthm, amssymb, booktabs, multirow, graphicx,
  booktabs, mathtools}
\usepackage[table]{xcolor}
\usepackage{dialogue}
\usepackage[stable]{footmisc}
\usepackage{tikz,nicefrac}
\usetikzlibrary{calc,matrix,arrows,decorations.markings}
\usepackage[colorlinks]{hyperref}
\usetikzlibrary{
  graphs,
  graphs.standard
}

\newtheorem{defin}{Definition}[section]

\newtheorem{theorem}[defin]{Theorem}

\newtheorem{lemma}[defin]{Lemma}

\newtheorem{example}[defin]{Example}

\newcommand{\R}{\mathbb{R}}
\newcommand{\Q}{\mathbb{Q}}

\newcommand{\Z}{\mathbb{Z}}

\DeclareMathOperator{\NP}{NP}

\newcommand{\cEHZ}{c_{\mathrm{EHZ}}}

\newcommand{\deltain}{\delta^{\mathrm{in}}}
\newcommand{\deltaout}{\delta^{\mathrm{out}}}
\newcommand{\permS}{\mathfrak{S}}

\begin{document}

\title{Computing the EHZ capacity is $\NP$-hard}

\author{Karla Leipold}
\author{Frank Vallentin}

\address{
  Karla Leipold,
  Universit\"at zu K\"oln, Department Mathematik/Informatik, Abteilung
  Mathematik, Weyertal 86-90, 50931 K\"oln,  Germany}
\email{karla.leipold@gmail.com}

\address{
  Frank Vallentin,
  Universit\"at zu K\"oln, Department Mathematik/Informatik, Abteil\-ung
  Mathematik, Weyertal 86-90, 50931 K\"oln,  Germany}
\email{frank.vallentin@uni-koeln.de}

\date{October 4, 2024}

\begin{abstract}
  The Ekeland-Hofer-Zehnder capacity (EHZ capacity) is a fundamental
  symplectic invariant of convex bodies.  We show that computing the
  EHZ capacity of polytopes is $\NP$-hard. For this we reduce the
  feedback arc set problem in bipartite tournaments to computing the
  EHZ capacity of simplices.
\end{abstract}

\maketitle

\markboth{K.~Leipold, F.~Vallentin}{Computing the EHZ capacity is $\NP$-hard}

\section{Introduction}
\label{sec:introduction}

Sympletic capacities are fundamental tools for measuring the
``symplectic size'' of subsets of the standard symplectic vector space
$(\R^{2n},\omega)$, cf.\ the book by Hofer,
Zehnder~\cite{Hofer-Zehnder-1994}. Many symplectic capacities,
starting with the Hofer-Zehnder capacity
from~\cite{Hofer-Zehnder-1994} and the Ekeland-Hofer capacity from
\cite{Ekeland-Hofer-1989} have been proved to coincide on convex
bodies and there common value is called the Ekeland-Hofer-Zehnder
capacity and denoted as~$\cEHZ$.

Our starting point is the following combinatorial formulation of the
EHZ capacity of convex polytopes due to
Haim-Kislev~\cite{Haim-Kislev-2019}. We assume that a $2n$-dimensional
convex polytope $P$ with $k$ facets is given by linear inequalities
\[
  P = P(B,c) = \{x \in \R^{2n} : Bx \leq c\} \quad \text{for} \quad B \in
  \R^{k \times (2n)},\; c \in \R^k.
\]
Let $b_1, \ldots, b_k \in \R^{2n}$ be the rows of $B$, written as
column vectors, and let $c_1, \ldots, c_k$ be the components of $c$.
So the polytope $P$ contains exactly the points $x \in \R^{2n}$ which
satisfy the linear inequalities $b_i^{\sf T} x \leq c_i$, with
$i = 1, \ldots, k$. Then the \emph{EHZ capacity} of $P$ is
\begin{equation}
  \label{eq:cEHZ-optimization-definition}
\begin{split}
  \cEHZ(P) = \frac{1}{2} \bigg(\max\bigg\{\sum_{1 \leq j < i \leq k}
 &  \beta_{\sigma(i)} \beta_{\sigma(j)} \, \omega(b_{\sigma(i)},
 b_{\sigma(j)}) \; : \;\\[-3ex]
&  \quad \sigma \in \permS_k,\; \beta \in \R^k_+,\; 
\beta^{\sf T} c = 1, \; \beta^{\sf T} B = 0\bigg\} \bigg)^{-1},
  \end{split}
\end{equation}
where $\omega$ is the standard symplectic form on $\R^{2n}$,  
\[
   \omega(x, y) = x^{\sf T} J y \quad
   \text{for} \quad x,y \in \R^{2n} \quad \text{and} \quad 
   J = \begin{bmatrix} 0 & I_n\\ -I_n & 0 \end{bmatrix},
\]
where $I_n$ is the identity matrix with $n$ rows and columns,
where $\permS_k$ denotes the symmetric group of $\{1, \ldots, k\}$, and
where $\R^k_+$ is the nonnegative orthant.

\smallskip

Hofer, Zehnder state~\cite[p. 102]{Hofer-Zehnder-1994} ``It turns out
that it is extremely difficult to compute this capacity.'' In this
paper we give a computational justification of this claim by showing
that computing the EHZ capacity is indeed $\NP$-hard, even for
the special case of simplices.

\begin{theorem}
  \label{thm:main}
  The following decision problem is $\NP$-complete: Given the rational inputs
  $\gamma \in \Q$, a matrix $B \in \Q^{k \times (2n)}$, a vector
  $c \in \Q^k$. Suppose that $P(B,c)$ is a
  $2n$-dimensional convex polytope having $k$ facets. Is the EHZ
  capacity of $P(B,c)$ at most $\gamma$; Is
  $\cEHZ(P(B,z)) \leq \gamma$?
\end{theorem}

To prove this theorem we shall reduce the feedback arc set problem in
bipartite tournaments to computing the EHZ capacity of a simplex, that
is to $2n$-dimensional polytopes having $k = 2n+1$ facets. The feedback
arc set problem in bipartite tournaments is known to be
$\NP$-complete; a result by Guo, H\"uffner, Moser
\cite{Guo-Hueffner-Moser-2007}.

\subsection{Notation from combinatorial optimization}

First we recall, for reference, some standard notation from
combinatorial optimization.  Let $D = (V, A)$ be a directed graph with
vertex set $V$ and with family of arcs $A$, allowing for multiple
arcs; see Schrijver \cite{Schrijver-2003}. Such a directed graph can
be represented by its adjacency matrix $M$ of size $|V| \times |V|$,
where the entry $M_{ij}$ equals the number of arcs from $v_i$ to $v_j$.

\smallskip

As usual we define a \emph{walk} in a directed graph to be a sequence
of arcs $a_1, \ldots, a_k \in A$ so that $a_i = (v_{i-1},v_i)$ with
$i = 1, \ldots, k$. A walk is \emph{closed} if \emph{start vertex} $v_0$
coincides with \emph{end vertex} $v_k$. A \emph{path} is a walk in which
vertices $v_0, \ldots, v_k$ are distinct. A \emph{directed circuit} is
a closed walk in which all vertices are distinct except that the start
vertex $v_0$ coincides with the end vertex $v_k$.

A directed graph is called \emph{acyclic} if it does not contain any
directed circuits.

For a subset of the vertex set $U \subseteq V$ we define
\[
  \deltaout(U) = \{(u,v) \in A : u \in U, v \in V \setminus
  U\}
\]
the multiset of arcs of $D$ leaving $U$, and similarly
\[
  \deltain(U) = \{(v,u) \in A : v \in V \setminus U, u \in U\}
\]
the arcs entering $U$. If $U$ consists only out of one vertex $v$ we
also write $\deltaout(v)$ instead of $\deltaout(\{v\})$, and
$\deltain(v)$ instead of $\deltain(\{v\})$. Then, the \emph{indegree}
of vertex $v$ is $|\deltain(v)|$ and its \emph{outdegree} is
$|\deltaout(v)|$, where we also count multiple arcs.

A walk in $D$ is called \emph{Eulerian} if each arc of $D$ is
traversed exactly once. The directed graph is called \emph{Eulerian}
whenever it has a closed Eulerian walk. A directed graph is Eulerian
if and only if
\begin{enumerate}
\item for any pair of vertices $u$, $v$ there is a path with start
  vertex $u$ and end vertex $v$; that is, $D$ is \emph{strongly connected},
\item the indegree of every vertex coincides with its outdegree.
\end{enumerate}

\smallskip

A \emph{feedback arc set} of $D$ is a subfamily of the arcs which
meets all directed circuits of $D$.  That is, removing these arcs from
$D$ would make $D$ acyclic.  The decision variant of the
\emph{feedback arc set problem} asks given a directed graph
$D = (V, A)$ and a natural number $\delta$, does $D$ have a feedback
arc set of cardinality at most $\delta$?

\smallskip

The \emph{maximum acyclic subgraph problem} is complementary to the
feedback arc set problem: Its optimization variant asks for the
maximum cardinality (using multiplicities) of arcs $A' \subseteq A$ so
that the resulting directed subgraph $(V, A')$ of $D$ is acyclic. A
subfamily $A'$ of $A$ defines an acyclic subgraph of $D$ if and only
if $A \setminus A'$ is a feedback arc set of $D$.

\smallskip

A \emph{bipartite tournament} is a directed graph $D = (U \cup V, A)$
whose vertex set is the disjoint union of vertices
$U = \{u_1, \ldots, u_n\}$ and $V = \{v_1, \ldots, v_m\}$ and in which
all of its arcs $A$ have the property that for every pair $i, j$, with
$i = 1, \ldots, n$, $j = 1, \ldots, m$, either $(u_i,v_j) \in A$ or
$(v_j,u_i) \in A$. So the underlying undirected graph is a complete
bipartite graph without multiple edges.

\subsection{Outline of the polynomial time reduction; Proof of
  Theorem~\ref{thm:main}}

Now we outline our polynomial time reduction, the details are provided
in the subsequent sections. We describe a polynomial time algorithm to
decide if a given bipartite tournament has a feedback arc set of at
most $\delta$ arcs by calling an oracle to solve the decision variant
of the EHZ-capacity of a simplex.  We will demonstrate that the feedback 
arc set of a complete bipartite tournament $D = (U \cup V, A)$ with $n = |U|$ and $m = |V|$, where we
assume that $n \geq m$ can be computed by the formula
\begin{equation}\label{mainformula}
|\tilde{A}|
- \frac{1}{2}
\left(\left\lfloor\frac{(2n+1)^2}
{2\cEHZ(P(\tilde{B},\textbf{e}))} + \frac{1}{2}
\right\rfloor+\Delta\right)-|\deltaout(x_{2n+1})|.
\end{equation} 
In the following we will explain this formula in detail. In particular we shall show that 
all appearing constants can be constructed in polynomial time. Hence, 
after solving \eqref{mainformula} for $\cEHZ(P(\tilde{B},\textbf{e}))$,
by the $\NP$-hardness of the
decision variant of the feedback arc set problem for bipartite
tournaments \cite{Guo-Hueffner-Moser-2007}, the decision variant of
the EHZ capacity of polytopes is also $\NP$-hard.

\smallskip

In Section~\ref{sec:from-feedback-arc-set-to-cEHZ} we transform the complete bipartite tournament 
$D$ into a simplex $P(\tilde{B},\textbf{e})$ through the block matrix
$B \in \Z^{(2n+1)\times (2n)}$ defined by
\[
    B = \begin{bmatrix}
    I_n & 0\\
    0 & S\\
    -\mathbf{e}^{\sf T} & -\mathbf{e}^{\sf T} S
  \end{bmatrix},
\]
where $\mathbf{e}^{\sf T} = (1 \, \ldots \, 1)$ is the all-ones (row)
vector, and $S \in \{-1,0,1\}^{n \times n}$ is defined componentwise
by
\begin{equation}\label{eq:simplex}
  S_{ij} = \begin{cases}
    1 & \text{if $(u_i, v_j) \in A$,}\\
    -1 & \text{if $(v_j, u_i) \in A$,}\\
    0 & \text{otherwise.}
    \end{cases}
\end{equation}
For simplicity we suppose in this outline that the matrix $S$ has full rank
$n$. If the matrix $S$ has not full rank, then we have to perform a
slight perturbation of it, yielding $\tilde{B}$; we give details about this technical issue
in Section~\ref{sec:from-feedback-arc-set-to-cEHZ}. It follows that
the polytope $P(B,\mathbf{e})$ is a simplex because the row vectors
$b_1, \ldots, b_{2n+1}$ of $B$ sum up to zero; that is $\sum_{i=1}^{2n+1}b_i=0$. 
This implies that the
variable $\beta \in \mathbb{R}^{2n+1}_+$ in the maximization
problem~\eqref{eq:cEHZ-optimization-definition} is determined
uniquely: $\beta = \frac{1}{2n+1} \mathbf{e}$. So the computation of
$\cEHZ(P(B,\mathbf{e}))$ simplifies to
\begin{equation}
  \label{eq:cEHZ-simplified}
\begin{split}
&  \cEHZ(P(B,\mathbf{e})) \\
& \;\;   = \; \;\frac{(2n+1)^2}{2} \bigg(\max\bigg\{\sum_{1 \leq j < i \leq 2n+1}
  \omega(b_{\sigma(i)},
  b_{\sigma(j)}) \; : \; \sigma \in \permS_{2n+1} \bigg\}\bigg)^{-1}.
\end{split}
\end{equation}

Since  $\omega(b_{\sigma(i)}, b_{\sigma(j)}) = b_{\sigma(i)}^{\sf T} J b_{\sigma(j)}$, 
in \eqref{eq:cEHZ-simplified} we are interested in simultaneously
permuting rows and columns of the matrix
$W = B J B^{\sf T} \in \mathbb{R}^{(2n+1) \times (2n+1)}$ so that the
sum of the entries of its lower triangular part ($i > j$) is
maximized. The matrix $W$ has the following block form
\[
  W =
\begin{bmatrix}
    I_n & 0\\
    0 & S\\
   -\mathbf{e}^{\sf T} & -\mathbf{e}^{\sf T} S\\
  \end{bmatrix}
  \,
  \begin{bmatrix}
    0 & I_n\\
    -I_n & 0\\
  \end{bmatrix}
  \,
  \begin{bmatrix}
I_n & 0 &   -\mathbf{e} &\\
0 & S^{\sf T} & -S^{\sf T} \mathbf{e}
\end{bmatrix}  
  =
  \begin{bmatrix}
0 & S^{\sf T} & -S^{\sf T}\textbf{e} \\
-S & 0 & S\textbf{e} \\
\textbf{e}^{\sf T}S & -\textbf{e}^{\sf T}S^{\sf T} & 0 \\
  \end{bmatrix}.
\]

\smallskip

In Section~\ref{sec:interpretation} we interpret the result of computing
$\cEHZ(P(B,\mathbf{e}))$ as given in \eqref{eq:cEHZ-simplified} in
terms of the feedback arc set problem. 
The matrix $W$ is skew-symmetric; $W^{\sf T} = -W$. Take the nonnegative part of $W$ to
define the matrix $M$; $M_{ij} = \max\{0,W_{ij}\}$. Then $W = M -
M^{\sf T}$ and for any permutation matrix $P_{\sigma}$, with $\sigma
\in \permS_{2n+1}$, we have
\begin{equation}\label{eq:WtoM}
  P_{\sigma}^{\sf T} W P_{\sigma} =   P_{\sigma}^{\sf T} (M - M^{\sf
    T}) P_{\sigma} = 2  P_{\sigma}^{\sf T} M P_{\sigma} -
  P_{\sigma}^{\sf T} (M + M^{\sf T}) P_{\sigma}.
\end{equation}
Hence, maximizing the sum of the entries of the lower triangular part
of $P_{\sigma}^{\sf T} W P_{\sigma}$ is equivalent to maximizing the
sum of the entries of the lower triangular part of
$P_{\sigma}^{\sf T} M P_{\sigma}$, because $M + M^{\sf T}$ is a
symmetric matrix and so the sum of the entries of the lower triangular
part of $P_{\sigma}^{\sf T} (M + M^{\sf T}) P_{\sigma}$ is independent
of the permutation $\sigma$. We denote this latter sum by $\Delta$, a constant which 
appears in \eqref{mainformula}.

The matrix $M$ defines the adjacency matrix of an auxiliary directed
graph on $2n+1$ vertices and $|\tilde{A}|$ arcs (appearing in~\eqref{mainformula}),
which is almost the original complete
bipartite tournament $D$. There are two essential differences between
the new auxiliary graph and the original graph $D$: A minor difference
occurs when $n > m$, then the new graph has $n-m$ isolated vertices,
which simply can be ignored. A major difference is that there is an
extra vertex $x_{2n+1}$ which is connected to some vertices in
$U \cup V$ by potentially multiple arcs. One feature of the
extra vertex is that it makes the auxiliary graph Eulerian.

It is known, the argument is recalled in
Section~\ref{sec:interpretation}, that maximizing the sum of the
entries of the lower triangular part of
$P_{\sigma}^{\sf T} M P_{\sigma}$ amounts to solving the maximum
acyclic subgraph problem on the auxiliary graph. Using this solution of
the maximum acyclic subgraph problem one can determine, by
complementarity, an optimal solution of the feedback arc set
problem of the auxiliary graph.

\smallskip

Then in Section~\ref{sec:cleaning-up} we relate this optimal solution
of the auxiliary graph to the feedback arc set problem of
the original graph: For this we shall use the fact
(Lemma~\ref{lem:Eulerian}) that feedback arc sets of the
original graph and of the auxiliary graph differ by exactly
$|\deltaout(x_{2n+1})|$ (appearing in~\eqref{mainformula}) many arcs. \hfill $\Box$

\section{From the feedback arc set problem to the EHZ capacity}
\label{sec:from-feedback-arc-set-to-cEHZ}

Let $D = (U \cup V, A)$ be a complete bipartite tournament with
$n = |U|$ and $m = |V|$, where we assume that $n \geq m$. Define the
matrix $S \in \{-1,0,1\}^{n \times n}$ componentwise as described in
\eqref{eq:simplex}.  If the matrix $S$ is not of full rank $n$ we
perform a small perturbation of the row vectors of $S$. For this let
$s_i$ be the $i$-th row of $S$ written as a column vector. Choose a
rational $\varepsilon > 0$. Then a desired perturbation can be
computed by the following steps, which are all easy, polynomial time,
linear algebra computations (for the details we refer, for instance,
to the book Schrijver \cite{Schrijver-1986}): Let $L \subseteq \Q^n$
be the subspace spanned by $s_1, \ldots, s_n$.

\begin{enumerate}
\item Among the vectors $s_1, \ldots, s_n$ choose a basis of
  $L$. Without loss of generality we may assume, after reordering,
  that $s_1, \ldots, s_k$ is such a basis.
\item Use Gram-Schmidt orthogonalization to compute an orthogonal
  basis of the orthogonal complement $L^{\perp}$.  Let
  $t_{k+1}, \ldots, t_n \in \Q^n$ be such an orthogonal basis
  normalized by $\|t_{k+1}\|_{\infty} = \ldots = \|t_{n}\|_{\infty} = 1$, where
  $\|\cdot\|_{\infty}$ denotes the $\ell_\infty$-norm.
\item Perturb the vectors $s_{k+1}, \ldots, s_n$ to
\[
  \tilde{s}_{k+1} =  s_{k+1}+\varepsilon
  t_{k+1}, \; \ldots, \; \tilde{s}_{n}= s_{n}+\varepsilon t_{n}.
\]
The first $k$ vectors stay unchanged: $\tilde{s}_1 = s_1, \ldots,
\tilde{s}_k = s_k$.
\end{enumerate}

In this way we get a perturbed matrix $\tilde{S}$ with rows
$\tilde{s}_1, \ldots, \tilde{s}_n$. From this we construct the block matrix
\[
  \tilde{B} =
  \begin{bmatrix}
I_n & 0\\
0 & \tilde{S}\\
-\mathbf{e}^{\sf T} & -\mathbf{e}^{\sf T} \tilde{S}
  \end{bmatrix}
\]
The convex hull of $\{\tilde{b}_1,\ldots, \tilde{b}_{2n+1}\}$, 
where $\tilde{b}_i$ are the row vectors of $\tilde{B}$, is a simplex with the origin in its interior;
by construction, the first $2n$ row vectors of $\tilde{B}$ are
linearly independent and all $2n+1$ row vectors sum up to zero. 
The polar of this simplex is equal to $P(\tilde{B},\textbf{e})$. 
Therefore, see, for instance, Ziegler \cite[Chapter 2.3]{Ziegler-1995}, $P(\tilde{B},\textbf{e})$
is a simplex.

\smallskip

In the following we will consider the EHZ capacity of this simplex
$P(\tilde{B},\mathbf{e})$. The principal computational task to
determine $\cEHZ(P(\tilde{B},\mathbf{e}))$ is to determine the maximum
\begin{equation}
   \label{eqn:simplex} 
  \max\bigg\{\sum_{1 \leq j < i \leq 2n+1} (P_\sigma^{\sf T} \tilde{W}  P_\sigma )_{ij}
   \; : \; \sigma \in \permS_{2n+1} \bigg\}
\end{equation}
for $\tilde{W} = \tilde{B} J \tilde{B}^{\sf T}$ and where $P_{\sigma}$
is the permutation matrix corresponding to the permutation
$\sigma \in \permS_{2n+1}$.  We have
\[
  \tilde{W} =
  \begin{bmatrix}
    0 & \tilde{S}^{\sf T} &  -\tilde{S}^{\sf T}\textbf{e}\\
    -\tilde{S} & 0 & \tilde{S}\textbf{e}\\
 \textbf{e}^{\sf T}\tilde{S} &-\textbf{e}^{\sf T}\tilde{S}^{\sf T} & 0\\
  \end{bmatrix}.
\]
Later on, for interpretation of the EHZ capacity in terms of an (integral) adjacency matrix, 
it is more convenient to work with $W$ rather than with $\tilde{W}$.
Since all entries of $\tilde{W}$ differ from $W$ only
by at most $n \varepsilon$, we can choose $\varepsilon = 1/n^4$ so that
\[
  \left\lfloor \sum_{1 \leq j < i \leq 2n+1} (P_\sigma^{\sf T}
    \tilde{W}  P_\sigma )_{ij} + \frac{1}{2} \right\rfloor =
  \sum_{1 \leq j < i \leq 2n+1} (P_\sigma^{\sf T}
    W  P_\sigma )_{ij}
  \]
for every $\sigma \in \permS_{2n+1}$.

\begin{example}
  \label{ex:example}
  We illustrate our construction for the following complete bipartite
  tournament.
  
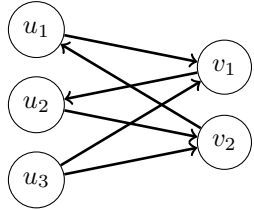
\begin{center}
  \begin{tikzpicture}
    \foreach \i in {1,2,3}
      \node[circle, draw, minimum size=0.2cm] (u\i) at (0,3-\i) {$u_\i$};
  
    \foreach \i in {1,2}
      \node[circle, draw, minimum size=0.2cm] (v\i) at (2.5,3-\i-0.5) {$v_\i$};
  
    \draw[->, line width=1pt] (u1) to (v1);
    \draw[->, line width=1pt] (v2) to (u1);
    \draw[->, line width=1pt] (v1) to (u2);
    \draw[->, line width=1pt] (u2) to (v2);
    \draw[->, line width=1pt] (u3) to (v1);
    \draw[->, line width=1pt] (u3) to (v2); 
  \end{tikzpicture}
  \end{center}
  Here
  \[
    s_1 = (1,-1,0), \; s_2= (-1,1,0), \; s_3 = (1,1,0),
  \]
  and
  \[
    \tilde{s}_1 = s_1, \; \tilde{s}_2 = s_2, \; \tilde{s}_3 = (1,1,\varepsilon),
  \]
  therefore
  \[
    \tilde{W} = \left(
      \begin{array}{ccc|ccc|c}
      0 & 0& 0 & 1 & -1 & 1 & -1\\
      0 & 0 & 0 & -1 & 1 & 1 & -1\\
      0 & 0 & 0 & 0 & 0 & \varepsilon & -\varepsilon\\\hline
      -1 & 1 & 0 & 0 & 0 & 0 & 0\\
      1 & -1 & 0 & 0 & 0 & 0 & 0\\
        -1 & -1 & -\varepsilon & 0 & 0  & 0 & 2+\varepsilon \\ \hline
       1 & 1 & \varepsilon & 0 & 0 & -2-\varepsilon & 0\\
      \end{array}
      \right).
  \]
  \end{example}

\section{Interpretation of the EHZ capacity}
\label{sec:interpretation}

In this section we interpret the maximization problem
\[
  \max\bigg\{\sum_{1 \leq j < i \leq 2n+1} (P_\sigma^{\sf T} W
  P_\sigma )_{ij} \; : \; \sigma \in \permS_{2n+1} \bigg\}
\]
in terms of the feedback arc set problem.

We define the matrix $M \in \R^{(2n+1) \times (2n+1)}$ by taking the
nonnegative part of $W$ so that $M_{ij} = \max\{0,W_{ij}\}$, and then
$W = M - M^{\sf T}$. For any permutation matrix $P_{\sigma}$, with
$\sigma \in \permS_{2n+1}$, we have
\[
  P_{\sigma}^{\sf T} W P_{\sigma} =   P_{\sigma}^{\sf T} (M - M^{\sf
    T}) P_{\sigma} = 2  P_{\sigma}^{\sf T} M P_{\sigma} -
  P_{\sigma}^{\sf T} (M + M^{\sf T}) P_{\sigma}.
\]
Hence, as noted in the introduction, maximizing the sum of the entries
of the lower triangular part of $P_{\sigma}^{\sf T} W P_{\sigma}$ is
equivalent to maximizing the sum of the entries of the lower
triangular part of $P_{\sigma}^{\sf T} M P_{\sigma}$.

\smallskip

The matrix $M$ defines the adjacency matrix of an auxiliary directed
graph $\tilde{D}$ on $2n+1$ vertices, which is, as the example below illustrates, related to $D$ by
\begin{enumerate}
\item identifying $x_1, \ldots, x_m$ with $v_1, \ldots, v_m$,
\item adding isolated vertices $x_{m+1}, \ldots, x_n$,
\item identifying $x_{n+1},\ldots, x_{2n}$ with $u_1, \ldots, u_n$,
\item adding the extra vertex $x_{2n+1}$,
\item reversing all arcs occorring in $D$.
\end{enumerate}

\begin{example} (first continuation of Example~\ref{ex:example})
  \begin{center}
    \begin{minipage}[t]{6cm}
      \vspace*{-4cm}
  \[
    M =
    \left(
      \begin{array}{ccc|ccc|c}
      0 & 0& 0 & 1 & 0 & 1 & 0\\
      0 & 0 & 0 & 0 & 1 & 1 & 0\\
      0 & 0 & 0 & 0 & 0 & 0 & 0 \\\hline
      0 & 1 & 0 & 0 & 0 & 0 & 0\\
      1 & 0 & 0 & 0 & 0 & 0 & 0\\
       0 & 0 & 0 & 0 & 0  & 0 & 2 \\ \hline
       1 & 1 & 0 & 0 & 0 & 0 & 0\\
      \end{array}
      \right)
    \]
  \end{minipage}
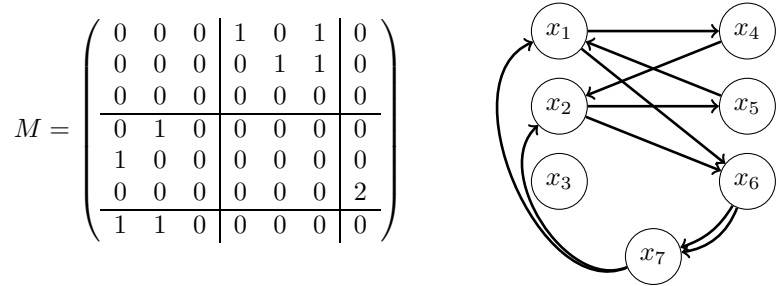
  \begin{minipage}[t]{6cm}
    \begin{tikzpicture}
    \foreach \i in {1,2,3}
      \node[circle, draw, minimum size=0.2cm] (x\i) at (0,4-\i) {$x_\i$};
  
    \foreach \i in {4,5,6}
      \node[circle, draw, minimum size=0.2cm] (x\i) at (2.5,7-\i) {$x_\i$};
  
     \node[circle, draw, minimum size=0.2cm] (x7) at (1.25,0) {$x_7$};
      
    \draw[->, line width=1pt] (x1) to (x4);
    \draw[->, line width=1pt] (x1) to (x6);
    \draw[->, line width=1pt] (x2) to (x5);
    \draw[->, line width=1pt] (x2) to (x6);
    \draw[->, line width=1pt] (x4) to (x2);
    \draw[->, line width=1pt] (x5) to (x1);
  
    \draw[->, line width=1pt] (x7) to [bend left=90] (x1);
    \draw[->, line width=1pt] (x7) to [bend left=80] (x2);
    \draw[->, line width=1pt] (x6) to [bend left = 20] (x7);
    \draw[->, line width=1pt] (x6) to [bend left = 30] (x7);
  \end{tikzpicture}
  \end{minipage}
  \end{center}
  \end{example}

Now we show that the maximization problem
\begin{equation}
\label{eq:MWAC}
  \max\bigg\{\sum_{1 \leq j < i \leq 2n+1} (P_\sigma^{\sf T} M
  P_\sigma )_{ij} \; : \; \sigma \in \permS_{2n+1} \bigg\}
\end{equation}
determines a maximum acyclic subgraph in $\tilde{D}$. This is
a simple, known fact; the argument uses the concept of topological
ordering (see, for instance, Knuth \cite{Knuth-1968}).

Let $D = (V, A)$ be a directed graph with $V = \{v_1, \ldots,
v_n\}$. A \emph{topological ordering} of $D$ is a permutation
$\sigma \in \permS_n$ so that
$(v_{i}, v_{j}) \in A$ only if
$\sigma^{-1}(i) < \sigma^{-1}(j)$, with $i, j = 1, \ldots, n$. That
is, the permutation $\sigma$ determines a ranking of the vertices so
that vertex $v_i$ becomes ranked at the $\sigma^{-1}(i)$-th
position. Then there is an arc between $v_i$ and $v_j$ only if vertex
$v_i$ is ranked before vertex $v_j$. There exists a topological
ordering of a directed graph if and only if the directed graph is
acyclic. Therefore, finding a maximum acyclic subgraph of $\tilde{D}$
is the same as finding a permutation $\sigma$ for which the number of
arcs in $\tilde{D}$ respecting the topological ordering given by
$\sigma$ is maximized, which is
\[
  \begin{split}
&  \max\bigg\{\sum_{\substack{i,j = 1, \ldots, n\\ \sigma^{-1}(i) <
    \sigma^{-1}(j)}} M_{ij} \; : \; \sigma \in \permS_{2n+1} \bigg\}
\\
 = \; & \max\bigg\{\sum_{1 \leq i < j \leq 2n+1} M_{\sigma(i),
   \sigma(j)} \; : \; \sigma \in \permS_{2n+1} \bigg\}.\\
  = \; & \max\bigg\{\sum_{1 \leq i < j \leq 2n+1} (P_\sigma^{\sf T} M
  P_\sigma)_{ij} \; : \; \sigma \in \permS_{2n+1} \bigg\}.
 \end{split}
\]
To maximize the sum of the lower triangular part instead of the sum of
the upper triangular part, that is summing over indices
$1 \leq j < i \leq 2n+1$ instead of $1 \leq i < j \leq 2n+1$, we
simply reverse all the arcs in $\tilde{D}$.  This global reversing operation
does not change the cardinality of the maximum acyclic subgraph of
$\tilde{D}$. Therefore, the maximization problem \eqref{eq:MWAC}
provides a solution of the maximum acyclic sugraph problem of
the auxilary graph $\tilde{D}$.

\begin{example} (second continuation of Example~\ref{ex:example}) A
  maximum acyclic subgraph of the auxiliary graph $\tilde{D}$
  is depicted by thick arcs. It has cardinality $7$. 
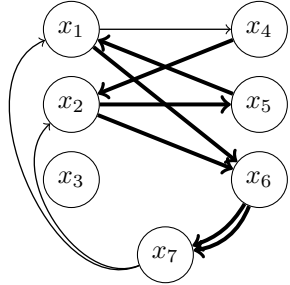
\begin{center}
\begin{tikzpicture}
  \foreach \i in {1,2,3}
    \node[circle, draw, minimum size=0.2cm] (x\i) at (0,4-\i) {$x_\i$};

  \foreach \i in {4,5,6}
    \node[circle, draw, minimum size=0.2cm] (x\i) at (2.5,7-\i) {$x_\i$};

   \node[circle, draw, minimum size=0.2cm] (x7) at (1.25,0) {$x_7$};
    
  \draw[->, line width=0.5pt] (x1) to (x4);
  \draw[->, line width=1.5pt] (x1) to (x6);
  \draw[->, line width=1.5pt] (x2) to (x5);
  \draw[->, line width=1.5pt] (x2) to (x6);
  \draw[->, line width=1.5pt] (x4) to (x2);
  \draw[->, line width=1.5pt] (x5) to (x1);

  \draw[->, line width=0.5pt] (x7) to [bend left=90] (x1);
  \draw[->, line width=0.5pt] (x7) to [bend left=80] (x2);
  \draw[->, line width=1.5pt] (x6) to [bend left = 20] (x7);
  \draw[->, line width=1.5pt] (x6) to [bend left = 30] (x7);
\end{tikzpicture}
\end{center}
The permutation
\[
  \sigma =
  \begin{pmatrix}
1 & 2 & 3 & 4 & 5 & 6 & 7 \\
4 & 2 & 7 & 1 & 3 & 5 & 6
\end{pmatrix}
\]
gives an optimal topological ordering with
\[
  P_\sigma^{\sf T} M P_\sigma =
\left(\;
  \begin{array}{ccccccccc}
    0 & \cellcolor{gray!20} 1 & \cellcolor{gray!20} 0 & \cellcolor{gray!20} 0 & \cellcolor{gray!20} 0 & \cellcolor{gray!20} 0 & \cellcolor{gray!20} 0 \\
    0 & 0 & \cellcolor{gray!20}  1 & \cellcolor{gray!20}  0 & \cellcolor{gray!20} 1 & \cellcolor{gray!20} 0 & \cellcolor{gray!20} 0 \\
    0 & 0 & 0 & \cellcolor{gray!20} 1 & \cellcolor{gray!20} 0 & \cellcolor{gray!20} 0 & \cellcolor{gray!20}  0 \\
    1 & 0 & 0 & 0 & \cellcolor{gray!20} 1 & \cellcolor{gray!20} 0 & \cellcolor{gray!20} 0 \\
    0 & 0 & 0 & 0 & 0 & \cellcolor{gray!20} 2 & \cellcolor{gray!20} 0 \\
    0 & 1 & 0 & 1 & 0 & 0 & \cellcolor{gray!20} 0 \\
    0 & 0 & 0 & 0 & 0 & 0 & 0 \\
\end{array}
\;\right).
\]
(Note that the permutation matrix $P_\sigma$ is acting on the rows.)
\end{example}

\section{Removing the extra vertex}
\label{sec:cleaning-up}

In the last step we show how the solution of the maximization
problem~\eqref{eq:MWAC} which determines the value of a maximum
acyclic subgraph of the auxiliary graph $\tilde{D}$ can be
used to solve the feedback arc set problem for the original graph $D$. 

If the extra vertex $x_{2n+1}$ is isolated, then there is nothing to do, otherwise we will
make use of the fact, that $\tilde{D}$ is Eulerian (after deleting the isolated vertices). To verify that $\tilde{D}$ is indeed Eulerian, 
one directly checks, for every vertex $x_i$, that $|\deltain(x_i)| = |\deltaout(x_i)|$ holds:
\[
\begin{split}
|\deltaout(x_i)| - |\deltain(x_i)| & =
\sum_{j=1}^{2n+1} M_{ij} - \sum_{j=1}^{2n+1} M_{ji}  
= e_i^{\sf T} M \mathbf{e} - \mathbf{e} M e_i^{\sf T} \\
& = e_i^{\sf T} (M - M^{\sf T}) \mathbf{e} = e_i^{\sf T} W \mathbf{e} = 0,
\end{split}
\]
where $e_i$ is the $i$-th standard basis vector in $\R^{2n+1}$.
Moreover, since $D$ is complete, it follows that $\tilde{D}$ is strongly connected.

\smallskip

The following argument is from Perrot, Van-Pham \cite[Section
3]{Perrot-Van-Pham-2015}. We provide the necessary details to make
this paper self-contained.

\begin{lemma}
\label{lem:Eulerian}
 The minimum cardinality of a feedback arc set in $D$
  equals the minimum cardinality of a feedback arc set in $\tilde{D}$
  minus $|\deltaout(x_{2n+1})|$.
\end{lemma}

\begin{proof}
We show the complementary result for the maximum acyclic subgraph
problem.

\smallskip

Let $A'$ be a maximum acyclic subgraph of $D$. Then we can use these
arcs together with the arcs $(x_{2n+1}, x_i)$, with
$i = 1, \ldots, 2n$, in $\tilde{D}$ leaving the extra vertex
$x_{2n+1}$ to determine an acyclic subgraph of $\tilde{D}$ of
cardinality $|A'| + |\deltaout(x_{2n+1})|$.

\smallskip

Conversely, let $A'$ be a maximum acyclic subgraph of $\tilde{D}$. Now
we describe a procedure to determine a maximum
acyclic subgraphs $A''$ so that the extra vertex $x_{2n+1}$ has no incoming
arc. Because $A''$ is maximum it has to contain all
$|\deltaout(x_{2n+1})|$ arcs leaving $x_{2n+1}$. Removing these arcs
determines an acyclic subgraph of the original graph $D$ of
cardinality $|A'| - |\deltaout(x_{2n+1})|$.

The procedure works as follows. We consider all vertices
$R$ which are reachable by walks from start
vertex $x_{2n+1}$ using only in arcs in $A'$. Define
\[
  A'' = \left(
A'
\setminus \deltain(R)
\right) \cup \deltaout(R).
\]
Then $A''$ determines again a maximum acyclic subgraph of
$\tilde{D}$ because by removing arcs in
$\deltain(R)$ and by adding arcs from
$\deltaout(R)$ we do not introduce
directed circuits. Furthermore, we have
\[
|A''| \geq  |A'| - |\deltain(R)| + |\deltaout(R)|,
\]
because $\deltaout(R)$, by the definition of
$R$, cannot contain arcs from $A'$. The
graph $\tilde{D}$ is Eulerian (ignoring the isolated vertices). So
$|\deltain(R)| = |\deltaout(R)|$ and hence $|A''| = |A'|$. The subset $A''$ cannot contain
edges from $\deltain(x_{2n+1})$ because $A'$ is acyclic.
\end{proof}

\begin{example} (third continuation of Example~\ref{ex:example}) The
  additional vertex $x_7$ has two incoming arcs in the maximum acyclic
  subgraph $A'$ of $\tilde{D}$.
  
Now we apply the procedure from the proof of Lemma~\ref{lem:Eulerian}
to determine a maximum acyclic subgraph in which $x_7$ has no
incoming arcs. Using only arcs in $A'$ the
only vertex reachable from $x_7$ is $x_7$ itself. So $R =
  \{x_7\}$ and
\[
  \deltaout(\{x_7\}) = \{(x_7,x_1), (x_7, x_2)\} \quad \text{and} \quad \deltain(\{x_7\})
  = \{(x_6,x_7), (x_6,x_7)\}
\]
and $A'' = (A' \setminus \deltain(\{x_7\})) \cup \deltaout(\{x_7\})$
determines the following maximum acyclic subgraph in which all
vertices but the isolated vertex $x_3$ are reachable from $x_7$.
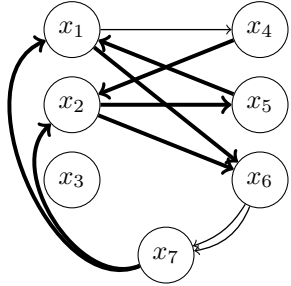
\begin{center}
\begin{tikzpicture}
  \foreach \i in {1,2,3}
    \node[circle, draw, minimum size=0.2cm] (x\i) at (0,4-\i) {$x_\i$};

  \foreach \i in {4,5,6}
    \node[circle, draw, minimum size=0.2cm] (x\i) at (2.5,7-\i) {$x_\i$};

   \node[circle, draw, minimum size=0.2cm] (x7) at (1.25,0) {$x_7$};
    
  \draw[->, line width=0.5pt] (x1) to (x4);
  \draw[->, line width=1.5pt] (x1) to (x6);
  \draw[->, line width=1.5pt] (x2) to (x5);
  \draw[->, line width=1.5pt] (x2) to (x6);
  \draw[->, line width=1.5pt] (x4) to (x2);
  \draw[->, line width=1.5pt] (x5) to (x1);

  \draw[->, line width=1.5pt] (x7) to [bend left=90] (x1);
  \draw[->, line width=1.5pt] (x7) to [bend left=80] (x2);
  \draw[->, line width=0.5pt] (x6) to [bend left = 20] (x7);
  \draw[->, line width=0.5pt] (x6) to [bend left = 30] (x7);
\end{tikzpicture}
\end{center}
\end{example}

\section*{Acknowledgements}

We like to thank Arne Heimendahl and Andreas Spomer for helpful
discussions. We also like to thank Alberto Abbondandolo for careful
reading the first version of this paper and for his many valuable
suggestions. We thank the anonymous referee for the helpful comments, 
suggestions, and corrections. The authors are partially supported by the SFB/TRR 191
``Symplectic Structures in Geometry, Algebra and Dynamics'' funded by
the DFG.


\begin{thebibliography}{[99]}

\bibitem{Ekeland-Hofer-1989}
I. Ekeland, E. Zehnder,
Symplectic topology and Hamiltonian dynamics,
Math. Z. 200 (1989), 355--378.

\bibitem{Guo-Hueffner-Moser-2007}
  J. Guo, F. H\"uffner, H. Moser,
  Feedback arc set in bipartite tournaments is $\NP$-complete,
  Information Processing Letters 102 (2007), 62--65.
  
\bibitem{Haim-Kislev-2019}
P. Haim-Kislev,
On the symplectic size of convex polytopes,
Geom. Funct. Anal. 29 (2019), 440--463.

\bibitem{Hofer-Zehnder-1994}
  H. Hofer, E. Zehnder,
  Symplectic invariants and Hamiltonian dynamics,
  Birkh\"auser, 1994.

\bibitem{Knuth-1968}
D.E. Knuth,
The art of computer programming, volume 1: Fundamental algorithms,
Addison-Wesley, 1968.

\bibitem{Perrot-Van-Pham-2015}
K. Perrot, T. Van-Pham,
Feedback arc set problem and NP-hardness of minimum recurrent
configuration problem of chip-firing game on directed graphs,
Ann. Comb. 19 (2015), 373--396.
  
\bibitem{Schrijver-1986}
  A. Schrijver,
  Theory of linear and integer programming,
  Wiley, 1986.

\bibitem{Schrijver-2003}
  A. Schrijver,
  Combinatorial optimization,
  Springer, 2003.

\bibitem{Ziegler-1995}
  G.M. Ziegler,
  Lectures on polytopes,
  Springer, 1995.

\end{thebibliography}
\end{document}